\documentclass[12pt,a4paper,oneside,notitlepage]{article}

\usepackage[T1]{fontenc}
\usepackage[cp1250]{inputenc}
\usepackage{amsfonts, amsmath, amssymb, amsthm}
\usepackage[english]{babel}
\usepackage[bookmarks,colorlinks]{hyperref}
\title{On nonlocal perturbations of integral kernels
\footnotetext{
The research was partially supported by grant MNiSW N N201 397137.}
}
\author{Krzysztof Bogdan
\footnote{Institute of Mathematics of the
Polish Academy of Sciences and 
Institute of Mathematics and Computer Science, Wroc{\l}aw University of Technology, Wybrze\.ze Wyspia\'nskiego 27, 50-370 Wroc{\l}aw, Poland,
bogdan@pwr.wroc.pl}
 \and Sebastian Sydor\footnote{Institute of Mathematics, University of Wroc{\l}aw, pl. Grunwaldzki 2/4, 50-384 Wroc{\l}aw, Poland, Sebastian.Sydor@math.uni.wroc.pl}}
\date{\today}

\newtheorem{theorem}{Theorem}[section]
\newtheorem{lemma}[theorem]{Lemma}

\newtheorem{corollary}[theorem]{Corollary}

\def \al {\alpha}

\def \si {\sigma}

\def \bbR {{\mathbb R}}
\def \bbN {{\mathbb N}}
\def \cM {{\cal M}}
\def \cB {{\cal B}}
\def \cE {{\cal E}}

\def \intl {\int\limits}
\def \iintl {\iint\limits}
\def \t {\widetilde}

\def \kK {{K}}
\def \kk {{k}}
\def \ka {{\kappa}}

\begin{document}
\maketitle
\begin{abstract}
We give sufficient conditions for a nonlocal perturbation of an integral kernel to be locally in time comparable with the kernel. 
\end{abstract}

\section{Introduction and Preliminaries}
We may delete or add jumps to a Markov process by adding a nonlocal operator to its generator.
We shall be concerned with estimates of the resulting, perturbed transition kernels.
In fact, we consider similar perturbations of rather general integral kernels on space-time. 
We focus on perturbations by non-local operators, which model evolution of mass in presence of births, deaths, dislocations and delays.
We are motivated by recent estimates of local, or Schr\"odinger, perturbations of integral kernels in \cite{MR3000465, MR3065313}, and nonlocal perturbations of the Green functions in \cite{MR2417435} and \cite{MR2316878}.

We deal with the so-called {\it forward} kernels, reflecting directionality of time.
The resulting perturbation and the original kernel turn out to be comparable locally in time and globally in space under an appropriate integral smallness condition on the first term of the perturbation series.
A related paper \cite{2011arXiv1112.3401C} studies nonlocal perturbations of the semigroup of the fractional Laplacian  and related discontinuous multiplicative and additive functionals, which offer a probabilistic counterpart of our approach. 
We emphasize that transition and potential kernels of Markov processes are our main motivation for this work, however in what follows we do not generally impose Chapman-Kolmogorov condition on the kernels. 

The paper is composed as follows. In Section~\ref{sec:mr} we formulate our main estimates: Theorem~\ref{t1} for kernels and Theorem~\ref{t2} and \ref{t3} for  kernel densities.
In Section~\ref{sec:tk} we note that nonlocal perturbations of transition kernels are transition kernels, too. In Section~\ref{sec:sp} we briefly mention signed perturbations and give lower bounds for negative perturbations of transition kernels.
In Section~\ref{sec:a} we indicate the extra work that needs to be done in order to verify our condition on the smallness of the first term of the perturbation series and we apply our results in specific situations. We focus on perturbations of the transition density of the fractional Laplacian, describe the perturbations in terms of generators and fundamental solutions and we illustrate the effect the nonlocal perturbations have on jump intensity of stochastic processes.

We note that Theorems~\ref{t1}, \ref{t2} and \ref{t3} generalize the main estimates of \cite{MR3000465} for Schr\"o\-din\-ger perturbations of integral kernels. 
The reader may find  in \cite{MR3000465} and a related paper \cite{MR3065313} some general comments on this research program, and more applications, e.g. to Weyl fractional integrals (\cite[Example~3]{MR3000465}) or to the potential kernel of the vector of two independent $1/2$-stable subordinators \cite[Example~4.1]{MR3065313}.

Considering transition probabilities, it should be noted that the perturbations considered in the present paper and \cite{2011arXiv1112.3401C} generally produce non-probabilistic kernels as they may increase the mass of the kernel. To preserve the mass, the generator of the perturbation should be of L\'evy-type; it should involve compensation, and annihilate constant functions. There is a considerable progress in construction and estimates of transition probabilities resulting from such operators. We refer the reader to
recent papers \cite{2013arXiv1306.5015Z}, \cite{2013arXiv1307.3087K} and \cite{2012arXiv1209.4096K},
whose techniques are close to perturbation methods, but  require
specific smoothness assumptions on the transition kernels.

\section{Main results}\label{sec:mr}
We first recall, after \cite{MR939365}, some properties of kernels.
Let $(E,\cE)$ be a measurable space. A kernel on $E$ is a map $K$ from $E\times\cE$ to $[0,\infty]$ such that
\begin{itemize}
\item[] $x\mapsto K(x,A)$ is $\cE$-measurable for all $A\in\cE$, and
\item[] $A\mapsto K(x,A)$ is countably additive for all $x\in E$.
\end{itemize}
\noindent Consider kernels $K$ and $J$ on $E$. The map
$$
(x,A)\mapsto \intl_E K(x,dy)J(y,A)
$$
from $(E\times\cE)$ to $[0,\infty]$ is another kernel on $E$, called the {\it composition} of $K$ and $J$, and denoted $KJ$. Here and below we alternatively write $\int f(x) \mu(dx)=\int \mu(dx)f(x)$.
We let $\kK _n=K_{n-1}JK(s,x,A)=(\kK J)^n\kK$, $n=0,1,\ldots$. The composition of kernels is associative, which yields the following lemma.
\begin{lemma}\label{l1}
$\kK _n=\kK _{n-1-m}J \kK _m$ for all $n\in\bbN$ and $m=0,1,\ldots,n-1$.
\end{lemma}
We define the {\it perturbation}, $\t{K}$, of ${K}$ by $J$, via the {\it perturbation series},
\begin{equation}\label{eq:deftk}
\t{K}=\sum_{n=0}^\infty \kK_n=\sum_{n=0}^\infty(KJ)^nK.
\end{equation}
Of course, $K\leq \t{K}$, and the following {\it perturbation formula} holds,
\begin{equation}\label{eq:pfdeftk}
\t{K}=K+\t{K}JK.
\end{equation}
Below we  prove upper bounds for $\t{K}$ under additional conditions on $K$, $J$ and $K_1=KJK$. 

Consider a set $X$ (the state space) with $\si$-algebra $\cM$, the real line  $\bbR$ (the time) equipped with the Borel sets $\cB_\bbR$, and consider the space-time
$$E:=\bbR \times X,$$ 
with the product $\si$-algebra ${\cE}=\cB_{\bbR}\times \cM$. 
Let $\eta \in [0,\infty)$ and a function $Q: \bbR \times \bbR \rightarrow [0, \infty)$ satisfy the following condition of {\it super-additivity}:
\begin{equation*}\label{a2}
Q(u,r)+Q(r,v)\leq Q(u,v)\quad \mbox{for all }u<r<v.
\end{equation*}
In particular, $Q(r,v)\leq Q(u,v)$ for $r\le u\le v$.
 Let $J$ be another kernel on $E$. We assume that  $K$ and $J$ are {\it forward} kernels, i.e. for $A\in \cE ,\; s\in \bbR, \, x\in X$,
\begin{equation*}\label{kernel}
\kK (s,x,A)=J(s,x,A)=0 \mbox{ whenever } A\subseteq(-\infty,s]\times X.
\end{equation*}
For $r<t$ we consider the strip $S=(r,t]\times X$, and the restriction of $K$ to $S$, to wit,
$K(s,x,A)$, where $(r,x)\in S$ and $A\subset S$.
We note that the restriction of $KJ$ to $S$ depends only on the restrictions of $K$ and $J$.
In fact we could consider $E=(r,t)\times X$ as our basic setting. This observation allows to localize our estimates in time.

In what follows we study consequences of the following assumption,
\begin{equation}\label{estJ_1}
KJK(s,x,A)\leq \intl_A [\eta + Q(s,t)]K(s,x,dtdy).
\end{equation}
\begin{theorem}\label{t1}
Assuming
\eqref{estJ_1},  for all $n=1,2,\ldots$, and $(s,x)\in E$, we have
\begin{align}
\kK _n(s,x,dtdy)&\leq \kK _{n-1}(s,x,dtdy)\left[\eta + \frac{Q(s,t)}{n}\right]\label{estJ_n}\\
&\leq \kK (s,x,dtdy)\prod_{l=1}^{n}\left[\eta + \frac{Q(s,t)}{l}\right].\label{e2}
\end{align}
If $0<\eta <1$, then for all $(s,x)\in E$,
\begin{equation}\label{eq:metK}
\t{\kK} (s,x,dtdy)\leq \kK (s,x,dtdy){\left(\frac{1}{1-\eta}\right)}^{1+Q(s,t)/\eta}.
\end{equation}
If $\eta =0$, then for all $(s,x)\in E$,
\begin{equation}\label{e4}
\t{\kK}(s,x,dtdy)\leq\kK (s,x,dtdy) e^{Q(s,t)}.
\end{equation}
\end{theorem}
\begin{proof}
(\ref{estJ_1}) yields (\ref{estJ_n}) for $n=1$. By induction, for $n=1,2,\ldots$ we have
\begin{eqnarray*}
&&(n+1)K_{n+1}(s,x,A)= nK_nJK(s,x,A)+K_{n-1}JK_1(s,x,A)\\
&=& n\intl_E K_n(s,x,dudz)(JK)(u,z,A)+ \intl_E (K_{n-1}J)(s,x,du_1dz_1)K_{1}(u_1,z_1,A)\\
&\leq& n\intl_E \left[\eta +\frac{Q(s,u)}{n}\right]K_{n-1}(s,x,dudz)(JK)(u,z,A)\\
&&+\intl_E (K_{n-1}J)(s,x,du_1dz_1)\intl_A [\eta + Q(u_1,t)]K(u_1,z_1,dtdy)\\
&=&(n+1)\eta K_n(s,x,A)\\
&&+\intl_E Q(s,u)K_{n-1}(s,x,dudz)\intl_E J(u,z,du_1 dz_1)\intl_A K(u_{1},z_{1},dtdy)\\ 
&&+\intl_E\intl_{(u,\infty)\times X} K_{n-1}(s,x,dudz)J(u,z,du_1dz_1)\intl_A Q(u_1,t) K(u_1,z_1,dtdy)\\
&\leq&(n+1)\eta K_n(s,x,A)\\
&&+\intl_A\intl_E\intl_E Q(s,u)K_{n-1}(s,x,dudz) J(u,z,du_1 dz_1) K(u_{1},z_{1},dtdy)\\ 
&&+\intl_A\intl_E\intl_E K_{n-1}(s,x,dudz)J(u,z,du_1dz_1) Q(u,t) K(u_1,z_1,dtdy)
\end{eqnarray*}
\begin{eqnarray*}
&\leq&(n+1)\eta K_n(s,x,A)\\
&&+\intl_A Q(s,t)\intl_E K_{n-1}(s,x,dudz) \intl_E J(u,z,du_1dz_1)K(u_1,z_1,dtdy)\\
&=&(n+1)\eta K_n(s,x,A)+\intl_AQ(s,t)\intl_EK_{n-1}(s,x,dudz)(JK)(u,z,dtdy)\\
&=&(n+1)\intl_A\left[\eta + \frac{Q(s,t)}{n+1}\right]K_{n}(s,x,dtdy).
\end{eqnarray*}
(\ref{e2}) follows from (\ref{estJ_n}), (\ref{e4}) results from Taylor\rq{}s expansion of the exponential function, and (\ref{eq:metK}) follows from the Taylor series
$$
(1-\eta)^{-a}=\sum_{n=0}^\infty \frac{\eta^n(a)_n}{n!},
$$
where $0<\eta<1$, $a\in \bbR$, and $(a)_n=a(a+1)\cdots(a+n-1)$.
\end{proof}

\noindent Theorem~\ref{t1} has two {\it finer} or {\it pointwise} variants, which we shall state under suitable conditions.
Fix a (nonnegative) $\si$-finite, non-atomic measure $$dt=\mu(dt)$$ on $(\bbR,\cB_\bbR)$ and a function $\kk(s,x,t,A)\ge 0$ defined for $s,t\in \bbR$, $x\in X$, $A \in \cM$, such that 
$k(s,x,t,dy)dt$ is a forward kernel and $(s,x)\mapsto k(s,x,t,A)$ is jointly measurable for all $t\in \bbR$ and $A\in \cM$. 
Let $\kk_0=\kk$, and for $n=1,2,\ldots$,
\begin{equation*}
\kk_n(s,x,t,A)=\intl_s^t\intl_X \kk_{n-1}(s,x,u,dz)\intl_{(u,t)\times X } J(u,z,du_1dz_1)\kk(u_1,z_1,t,A)du.
\end{equation*}
The perturbation, $\t{k}$, of $k$ by $J$, is defined as
\begin{equation*}
\t{k}=\sum_{n=0}^{\infty}k_n.
\end{equation*}
Assume that
\begin{equation*}\label{a3}
\intl_s^t\intl_X \kk(s,x,u,dz)\intl_{(u,t)\times X } J(u,z,du_1dz_1)\kk(u_1,z_1,t,A)du \leq [\eta+Q(s,t)]\kk(s,x,t,A).
\end{equation*}
\begin{theorem}\label{t2}
Under the assumptions, for all $n=1,2,\ldots$, and $(s,x)\in E$,
\begin{align*}
\kk _n(s,x,t,dy)&\leq \kk _{n-1}(s,x,t,dy)\left[\eta + \frac{Q(s,t)}{n}\right]\\
&\leq \kk (s,x,t,dy)\prod_{l=1}^{n}\left[\eta + \frac{Q(s,t)}{l}\right].
\end{align*}
If $0<\eta <1$, then for all $(s,x)\in E$ and $t\in \bbR$ we have
\begin{equation*}
\t{\kk} (s,x,t,dy)\leq \kk (s,x,t,dy){\left(\frac{1}{1-\eta}\right)}^{1+Q(s,t)/\eta}.
\end{equation*}
If $\eta =0$, then 
\begin{equation*}
\t{\kk}(s,x,t,dy)\leq\kk (s,x,t,dy) e^{Q(s,t)}.
\end{equation*}
\end{theorem}
We skip the proof, because it is similar to the proof of Theorem~\ref{t1}.

\noindent
For the {\it finest} variant of Theorem~\ref{t1}, we fix a 
$\si$-finite measure 
$$dz=m(dz)$$ 
on $(X, \cM)$.
We consider function $\ka(s,x,t,y)\geq 0$, $s,t\in \bbR$, $x,y\in X$, such that
$\ka(s,x,t,y)dtdy$ is a forward kernel and $(s,x)\mapsto k(s,x,t,y)$ is jointly measurable for all $t\in \bbR$ and $y\in X$.
We call such $\ka$ a (forward) {\it kernel density} (see \cite{MR3000465}).
We define $\ka_0(s,x,t,y) = \ka(s,x,t,y)$, and
$$
\ka_n(s,x,t,y) = \intl_s^t \intl_X \ka_{n-1}(s,x,u,z)\intl_{(u,t)\times X}J(u,z,du_1dz_1)\ka(u_1,z_1,t,y) \,dz\,du\,,
$$
where $n=1,2,\ldots$. Let $\t{\ka}=\sum_{n=0}^\infty \ka_n$. For all $s<t\in \bbR$, $x,y\in X$, we assume
\begin{equation*}
\!\intl_s^t\!\!\intl_X\!\!\ka(s,x,u,z)\!\!\!\intl_{(u,t)\times X}\!\!J(u,z,du_1dz_1)\ka(u_1,z_1,t,y)dz du\leq [\eta +Q(s,t)]\ka(s,x,t,y).
\end{equation*}
\begin{theorem}\label{t3}
Under the assumptions, for $n=1,2,\ldots$, $s<t$ and $x,y \in X$,
\begin{align*}
\ka _n(s,x,t,y)&\leq \ka _{n-1}(s,x,t,y)\left[\eta + \frac{Q(s,t)}{n}\right]\\
&\leq \ka (s,x,t,y)\prod_{l=1}^{n}\left[\eta + \frac{Q(s,t)}{l}\right].
\end{align*}
If $0<\eta <1$, then for all $s,t\in \bbR$ and $x,y \in X$, 
\begin{equation*}\label{eq:metgKpt}
\t{\ka} (s,x,t,y)\leq \ka (s,x,t,y){\left(\frac{1}{1-\eta}\right)}^{1+Q(s,t)/\eta}.
\end{equation*}
If $\eta =0$, then 
\begin{equation*}\label{eq:gpte4}
\t{\ka}(s,x,t,y)\leq \ka (s,x,t,y)e^{Q(s,t)}.
\end{equation*}
\end{theorem}
We also skip this proof, because it is similar 
that of Theorem~\ref{t1}.
\section{Transition kernels}\label{sec:tk}
Let $k$ above (note the joint measurability) be a {\it transition kernel} i.e. additionally satisfy the Chapman-Kolmogorov conditions for $s<u<t$, $A\in \cM$, 
\begin{equation*}\label{eq:CK}
\int_X k(s,x,u,dz)k(u,z,t,A)=k(s,x,t,A).
\end{equation*}
We note that we do {\it not} assume $k(s,x,t,X)=1$.

Following \cite{MR2457489}, we shall show that $\t{k}$ is a transition kernel, too.
\begin{lemma}\label{le:prop:1}
For all $s<u<t$, $x,y\in X$, $A\in \cM$ and $n=0,1,\ldots$,
\begin{equation}\label{prop:1}
\sum_{m=0}^n\intl_X k_m(s,x,u,dz)k_{n-m}(u,z,t,A)=k_n(s,x,t,A)
\end{equation}
\end{lemma}
\begin{proof}
We note that (\ref{prop:1}) is true for $n=0$ by fact that $k$ is a transition kernel and satisfies the Chapman-Kolmogorov equation. Assume that $n\geq 1$ and (\ref{prop:1}) holds for $n-1$. The sum of the first $n$ terms on the left of (\ref{prop:1}) can be dealt with by induction:
\begin{eqnarray}\label{prop:11}
&&\sum_{m=0}^{n-1}\intl_X k_m(s,x,u,dz)k_{n-m}(u,z,t,A)\\
&=&\sum_{m=0}^{n-1}\intl_X k_m(s,x,u,dz)\intl_u^t \intl_X k_{n-m-1}(u,z,r,dw) \nonumber\\
&&\times\intl_{(r,\infty)\times X}J(r,w,dr_1dw_1)k(r_1,w_1,t,A)dr\nonumber\\
&=&\intl_u^t \intl_X \intl_{(r,\infty)\times X} J(r,w,dr_1dw_1)k(r_1,w_1,t,A)\nonumber\\
&&\times \sum_{m=0}^{n-1} \intl_X k_m(s,x,u,dz)k_{(n-1)-m}(u,z,r,dw)dr\nonumber\\
&=&\intl_u^t \intl_X k_{n-1}(s,x,r,dw)\intl_{(r,\infty)\times X} J(r,w,dr_1dw_1)k(r_1,w_1,t,A)dr.\nonumber
\end{eqnarray} 
The $(n+1)$-st term on the left of \eqref{prop:1} is
\begin{eqnarray}\label{prop:12}
&&\intl_X k_n(s,x,u,dz)k(u,z,t,A)\\
&=&\intl_X \intl_s^u\intl_X k_{n-1}(s,x,r,dw)\intl_ {(r,\infty)\times X} J(r,w,dr_1dw_1) k(r_1,w_1,u,dz)k(u,z,t,A)dr\nonumber\\
&=&\intl_s^u\intl_X k_{n-1}(s,x,r,dw)\intl_{(r,\infty)\times X} J(r,w,dr_1dw_1)k(r_1,w_1,t,A)dr,\nonumber
\end{eqnarray}
and (\ref{prop:1}) follows on adding (\ref{prop:11}) and (\ref{prop:12}).
\end{proof}
\begin{lemma}
For all $s<u<t$, $x,y\in \bbR^d$ and $A\in \cM$,
\begin{equation*}\label{chapman}
\intl_X\t{k}(s,x,u,dz)\t{k}(u,z,t,A)=\t{k}(s,x,t,A).
\end{equation*}
\end{lemma}
We refer to \cite[Lemma~2]{MR2457489} for the proof, based on \eqref{prop:1}.
Thus, $\t{k}$ is a transition kernel.

Similarly, the function  $\ka$ considered above  (note the joint measurability) is called transition density if it satisfies Chapman-Kolmogorov equations pointwise.
In an analogous way we then prove that $\t{\ka}$ defined above is a transition density, provided so is $\ka$. 

\section{Signed perturbation}\label{sec:sp}
The following discussion is modeled after \cite{MR2457489}.
We  consider perturbation of $K$ by  $m(s,x,t,y)J(s,x,dtdy)$, where $m: \bbR\times X\times \bbR\times X\to [-1,1]$ is jointly measurable. If $\t{K}$, our perturbation of $K$ by $J$, is finite, then the 
perturbation series resulting from $mJ$ is absolutely convergent, and the perturbation formula extends to this case.
For instance, 
the perturbation of $K$ by $-J$ is 
\begin{equation*}
\t{K}^-=\sum_{n=0}^{\infty}(-1)^n (KJ)^nK,
\end{equation*}
and
\begin{equation*}\label{eq:pfdeftk-}
\t{K}^-=K-\t{K}^-JK.
\end{equation*}
Clearly, if $\t{K}^-\ge 0$, then $\t{K}^-\le K$, but the former property is delicate cf. \cite[Section~4]{MR2457489}.
In this connection we note that if  $K$ is restricted to $S=(s,t]\times X$, then
under the assumptions of Theorem \ref{t1} by \eqref{estJ_n} we have (on $S$)
\begin{eqnarray*}
\t{K}^-&=&[K-KJK]+[(KJ)^2K-(KJ)^3K]-\ldots\\
&\ge& \sum_{n=0,2,\ldots}\left(1-\eta-\frac{Q(s,t)}{n+1}\right)(KJ)^{n}K
\geq \frac{1-\eta}{2}K,
\end{eqnarray*}
provided $Q(s,t)\le (1-\eta)/2$ 
and we also have (on $S$)
\begin{eqnarray}
\t{K}^-&=&K-[KJK-(KJ)^2K]-[(KJ)^3K-(KJ)^4K]-\ldots\nonumber\\
&\leq& K-\sum_{n=1,3,\ldots}\left(1-\eta-\frac{Q(s,t)}{n+1}\right)(KJ)^{n}K
\leq K,\label{eq:pfdeftk-ub}
\end{eqnarray}
provided $Q(s,t)\leq 2(1-\eta)$.
Chapman-Kolmogorov equations allow to propagate this for {\it transition} kernels $k$ as follows.
If 
$s=u_0<u_1<\ldots<u_{n-1}<u_n=t$ and 
$Q(u_{l-1},u_l)\le (1-\eta)/2$ for $l=1,2,\ldots,n$, then
\begin{eqnarray}
\nonumber
&&\t{\kk}(s,x,t,A)=\intl_X\ldots \intl_X \t{\kk}(s,x,u_1,dz_1)\t{\kk}(u_1,z_1,u_2,dz_2)\ldots \t{\kk}(u_{n-1},z_{n-1},t,A)\\\nonumber
&\geq & \left(\frac{1-\eta}{2}\right)^{n}\intl_X\ldots \intl_X \kk(s,x,u_1,dz_1)\kk(u_1,z_1,u_2,dz_2)\ldots \kk(u_{n-1},z_{n-1},t,A)\\
&=&\left(\frac{1-\eta}{2}\right)^{n} \kk(s,x,t,A).\label{eq:lb2n}
\end{eqnarray}
If $Q(s,t)\leq h(t-s)$ for a function $h$, and $h(0^+)=0$, then global nonnegativity and lower bounds for $\t{\kk}^-$ easily follow, and so 
\begin{equation*}
0\le \t{k}^-\le k.
\end{equation*}
Analogous results hold pointwise for {\it transition} densities $\kappa$ (we skip details).

We remark that estimates of transition kernels give bounds for the corresponding resolvent and potential operators provided we also have bounds for large times (see \cite[Lemma~7]{KBTJ2012} and \eqref{eq:bsg} in this connection).

\section{Applications}\label{sec:a}
Verification of \eqref{estJ_1} usually requires some work. Here is a case study.
Let $\al\in (0,2)$. Consider the convolution semigroup of functions defined as
\begin{equation}\label{eq:tFpt}
p_t(x)=(2\pi)^{-d}\int_{\bbR^d}e^{ixu}e^{-t|u|^{\al}}du\quad {\rm for }\quad t>0\; x\in \bbR^d.
\end{equation}
The semigroup is generated by the fractional Laplacian $\Delta^{\alpha/2}$ (\cite{MR2569321}).
By \eqref{eq:tFpt},
\begin{equation*}
p_t(x)= t^{-\frac{d}{\al}}p_1(t^{-\frac{1}{\al}}x).\label{scaling}
\end{equation*}
By subordination (\cite{MR2569321}) we see that $p_t(x)$ is decreasing in $|x|$:
\begin{equation}
p_t(x)\geq p_t(y)\quad {\rm if } \quad |x|\leq |y|.\label{decreasing}
\end{equation}
We write $f(a,\ldots,z)\approx g(a,\ldots,z)$ if there is a number $0<C<\infty$ independent of $a,\ldots,z$, i.e. a {\it constant}, such that $C^{-1}f(a,\ldots,z)\leq g(a,\ldots,z) \leq Cf(a,\ldots,z)$ for all $a,\ldots,z$.
We have (see, e.g., \cite{KBTJ2012}),
\begin{equation}\label{property}
p_t(x) \approx t^{-\frac{d}{\al}}\wedge\frac{t}{|x|^{d+\al}}.
\end{equation}
\noindent Noteworthy, 
$t^{-\frac{d}{\al}}\leq{t}/|x|^{d+\al}$ iff $t\leq |x|^{\al}$.
We observe the following property:
\begin{equation*}
{\rm If} \quad |x|\approx|y|, \quad {\rm then}\quad p_t(x)\approx p_t(y).
\end{equation*}
We denote
\begin{equation*}
p(s, x, t, y) = p_{t-s}(y - x),\quad  x, y \in \bbR^d,  s < t\label{notation}.
\end{equation*}
This $p$ is the transition density of the standard isotropic $\al$-stable L\'evy process $(Y_t, P^x)$ in $\bbR^d$ with the L\'evy measure $\nu(dz)=c|z|^{-d-\alpha}dz$, and generator $\Delta^{\alpha/2}$.

To study \eqref{estJ_1}, we consider nonnegative jointly Borelian $j(x,y)$ on $\bbR^d\times \bbR^d$, and we define the norm 
\begin{equation*}\label{norm}
\|j\| :=\left( \sup_{z\in \bbR^d}\intl_{\bbR^d}|j(z,w)|dw\right) \vee \left(\sup_{w\in \bbR^d}\intl_{\bbR^d}|j(z,w)|dz \right).
\end{equation*} 
\begin{lemma}\label{l:con_j}
There are $\eta \in [0,1)$ and $c<\infty$ such that
\begin{equation}\label{condition}
\!\intl_s^t\! du\!\intl_{\bbR^d}\! dz\!\intl_{\bbR^d}\! dw\ p(s,x,u,z)j(z,w)p(u,w,t,y)\leq[\eta+c(t-s)]p(s,x,t,y),
\end{equation}
if $\|j\|<\infty$, $|j(z,w)|\le \varepsilon|w-z|^{-d-\al}$ and  $\varepsilon>0$ is sufficiently small.
\end{lemma}
\begin{proof}   
Denote $I=p(s,x,u,z)j(z,w)p(u,w,t,y)$. Consider three sets $A_1=\{(z,w)\in \bbR^d \times \bbR^d:|z-y|\leq 4\}$, $A_2=\{(z,w)\in \bbR^d \times \bbR^d:|w-x|\leq 4|z-x|\}$ and $B=\{(z,w)\in \bbR^d \times \bbR^d:|z-x|\leq \frac{1}{3}|y-x|, \; |w-y|\leq \frac{1}{3}|y-x|\}$. The union of $A_1, A_2$ and $B$ gives the whole of $\bbR^d$.

If $|z-y|\leq 4|w-y|$, then 
$p(u,w,t,y)\leq c_1 p(u,z,t,y)$, and by (\ref{decreasing}),
\begin{eqnarray*}
\intl_s^t du\iintl_{A_1}dzdw\; I &{\leq}& c_1 \intl_s^t du\iintl_{A_1}dzdw \;p(s,x,u,z)j(z,w)p(u,z,t,y) \\
&{\leq}& c_1 \|j\|\intl_s^t du\intl_{\bbR^d}dz \;p(s,x,u,z)p(u,z,t,y)\\
&=& c_1 \|j\|(t-s)p(s,x,t,y),
\end{eqnarray*}
which is satisfactory, see \eqref{l:con_j}.
The case of $A_2$ is similar.
For $B$ we first consider the case $t-s \leq 2|y-x|^{\al}$, and we obtain
\begin{eqnarray*}
\intl_s^t du\iintl_{B}dzdw\; I &{\leq}& \intl_s^t du\iintl_{B}dzdw\; p(s,x,u,z)\varepsilon |w-z|^{-d-\al}p(u,w,t,y),
\end{eqnarray*}
\begin{eqnarray*}
&{\leq}& 3^{d+\al}\varepsilon \intl_s^t du\iintl_{B}dzdw\; p(s,x,u,z)|y-x|^{-d-\al} p(u,w,t,y)\\
&\leq& 3^{d+\al}\varepsilon \intl_s^t du\intl_{\bbR^d}\intl_{\bbR^d}dzdw\; p(s,x,u,z)p(u,w,t,y)\\
&=&3^{d+\al}\varepsilon|y-x|^{-d-\al}(t-s) \approx 3^{d+\al}\varepsilon p(s,x,t,y).
\end{eqnarray*}
In the case $t-s > 2|y-x|^{\al}$ we obtain
\begin{eqnarray*}
\intl_s^t du\iintl_{B}dzdw\; I &=&\intl_s^{\frac{s+t}{2}} du\iintl_{B}dzdw\; p(s,x,u,z)j(z,w)p(u,w,t,y)\\
&&+\intl_{\frac{s+t}{2}}^t du\iintl_{B}dzdw\; p(s,x,u,z)j(z,w)p(u,w,t,y)\\
&\leq&\intl_s^{\frac{s+t}{2}} du\iintl_{B}dzdw\; p(s,x,u,z)j(z,w) (t-u)^{-\frac{d}{\al}}\\
&&+\intl_{\frac{s+t}{2}}^t du\iintl_{B}dzdw\; (u-s)^{-\frac{d}{\al}}j(z,w)p(u,w,t,y)\\
&\leq&\intl_s^{\frac{s+t}{2}} du\iintl_{B}dzdw\; p(s,x,u,z)j(z,w) \left(\frac{t-s}{2}\right)^{-\frac{d}{\al}}\\
&&+\intl_{\frac{s+t}{2}}^t du\iintl_{B}dzdw\; \left(\frac{t-s}{2}\right)^{-\frac{d}{\al}}j(z,w)p(u,w,t,y)\\
&\leq& 2^{\frac{d}{\al}}\|j\|(t-s)^{-\frac{d}{\al}}(t-s)\approx 2^{\frac{d}{\al}}\|j\|(t-s)p(s,x,t,y).
\end{eqnarray*}
We can take $\eta=3^{d+\al}\varepsilon$ and $c=c_1\|j\|+2^{d/\al}\|j\|$ in \eqref{condition}.
\end{proof}
In what follows, $\t{p}$  denotes the perturbation of $p$ by $J(s,x,dtdy)=j(x,y)\delta_s(dt)dy$, and $\t{p}^{\,-}$ is the perturbation of $p$ by $-J$.
In view of Theorem~\ref{t3} and \eqref{eq:lb2n} we obtain the following result.
\begin{corollary}\label{cor:lupfL}
If \eqref{condition} holds with $0\le\eta <1$, then for $s,t\in \bbR$, $x,y\in \bbR^d$,
\begin{equation}\label{eq:ueb}
\t{p} (s,x,t,y)\leq p (s,x,t,y){\left(\frac{1}{1-\eta}\right)}^{1+c(t-s)/\eta},
\end{equation}
and 
\begin{equation*}
p (s,x,t,y)\left(\frac{1-\eta}{2}\right)^{1+2c(t-s)/(1-\eta)}\le \t{p}^{\,-} (s,x,t,y)\le p (s,x,t,y).
\end{equation*}
\end{corollary}
\noindent
If $j(z,w)=j(w,z)$, then the estimates agree with those obtained in \cite{MR2008600}.

We shall verify that $\t{p}$ is the fundamental solution of $\Delta^{\alpha/2}+J$, i.e.
\begin{equation}\label{fundt:1}
\intl_{\bbR}\intl_{\bbR^d}\t{p}(s,x,t,y)[\partial_t+\Delta_y^{\al/2}+j(x,y)]\phi(t,y)dydt=-\phi(s,x),
\end{equation}
provided \eqref{condition} holds with $0\le\eta <1$. Here and below $s\in\bbR $, $x\in\bbR^d$, and $\phi$ is a smooth compactly supported function on $\bbR\times\bbR^d$.
By \eqref{eq:tFpt} (see also \cite{KBTJ2012}),
\begin{equation}\label{fund:1}
\intl_{\bbR}\intl_{\bbR^d}p(s,x,t,y)[\partial_t+\Delta_y^{\al/2}]\phi(t,y)dydt=-\phi(s,x).
\end{equation}
We denote $P(s,x,dt,dy)=p(s,s,t,y)dtdy$, $(L\phi)(s,x)=\partial_t\phi(s,x)+\Delta_y^{\al/2}\phi(s,x)$
 and $\t{P}(s,x,dt,dy)=\t{p}(s,x,t,y)dtdy$.
By (\ref{fund:1}),
$PL\phi=-\phi$.
By \eqref{eq:deftk} and \eqref{eq:ueb}, 
\begin{equation}\label{fund:4}
\t{P}(L+J)\phi=PL\phi+\sum_{n=1}^{\infty}(PJ)^nPL\phi+ \sum_{n=0}^{\infty}(PJ)^{n+1}\phi=-\phi,
\end{equation}
where the series converge absolutely.
This proves (\ref{fundt:1}). We see that the argument is quite general, and hinges only on the convergence of the series (see \cite[Lemma 4 and 5]{MR3000465} for more insight).

We now return to the setting of Theorem~\ref{t2} to illustrate the influence of the perturbation on jump intensity of Markov processes. We consider
$k$ being the transition probability of a L\'evy process $(X_t)_{t\ge 0}$ on $\bbR^d$ (\cite{MR1739520}). Let $\nu(dy)$ be the L\'evy measure, i.e. the jump intensity of $(X_t)$. We have $k(s,x,t,A) =\varrho_{t-s}(A-x)$, where $t>s$ and $\varrho_{t}$ is the distribution of $X_t$.
Let $\mu$ be a finite measure on $\bbR^d$ and $J(s,x,dtdy)=\mu(dy-x)\delta_s(dt)$ for $s<t$. By induction we verify that
$$
k_n(s,x,t,dy)=\frac{(t-s)^n}{n!}\varrho_{t-s}\ast \mu^{\ast n}(dy-x).
$$
Therefore, 
$$
\t{k}(s,x,t,dy)=\varrho_{t-s}\ast \sum_{n=0}^{\infty}\frac{(t-s)^n}{n!} \mu^{\ast n}(dy-x)
$$
cf. \cite{2011arXiv1112.3401C},
and so
\begin{equation}\label{eq:bsg}
e^{-(t-s)|\mu|}\t{k}(s,x,t,dy)
\end{equation}
is the transition probability of a L\'evy process with the L\'evy measure $\nu+\mu$.
Thus, perturbing $k$ by $J$ adds jumps and some mass to $(X_t)$, and perturbing by $-J$ reduces jumps and mass of $(X_t)$, as long as $\nu-\mu$ is nonnegative. This is sometimes called Meyer's procedure of adding/removing jumps in probability literature.

We like to note that subtracting jumps may destroy our (local in time, global in space) comparability of $k$ and $\t{k}^-$. 
Indeed, 
we can make $\nu(dz)-\mu(dz)$ a compactly supported L\'evy measure, whose transition probability has a different, superexponential decay in space (compare \cite[Lemma~2]{MR2794975} and \eqref{property}). This explains the role played by the smallness assumption on $\varepsilon$ in Lemma~\ref{l:con_j} and Corollary~\ref{cor:lupfL}.

{\bf Acknowledgements.} We thank Tomasz Jakubowski for discussions and suggestions.

\end{document}